\documentclass[a4paper,12pt]{article}
\usepackage{amsmath}
\usepackage{amssymb}
\usepackage{latexsym}
\usepackage{amsthm}
\newtheorem{Thm}{Theorem}[section]

\newtheorem{Lem}[Thm]{Lemma}
\newtheorem{Prop}[Thm]{Proposition}

\theoremstyle{definition}

\newtheorem{Ass}[Thm]{Assumption}

\newtheorem{Rem}[Thm]{Remark}

\begin{document}

\title{Large deviations for simple random walk on percolations with long-range correlations}
\author{Kazuki Okamura\footnote{Graduate School of Mathematical Sciences, The University of Tokyo  e-mail : \texttt{kazukio@ms.u-tokyo.ac.jp}  \, \, This work was supported by Grant-in-Aid for JSPS Fellows. }}
\date{}
\maketitle

\begin{abstract}
We show quenched large deviations for the simple random walk on a certain class of percolations with long-range correlations.
This class contains  the supercritical Bernoulli percolations, 
the model considered by Drewitz, R\'ath, and Sapozhnikov and the random-cluster model up to the slab critical point.  
Our result is an extension of Kubota's result for the supercritical Bernoulli percolations. 
We also state a shape theorem for the chemical distance,
which is an extension of Garet and Marchand's result for the supercritical Bernoulli percolations.
\end{abstract}

\section{Introduction and Main result}

%Introduction

%Intro
In the research of percolation,
it is important to understand geometric properties of clusters and  behaviors of random walks on the clusters.
In the case of the supercritical Bernoulli percolations, 
Antal and Pisztora \cite{AP} gave 
large deviation estimates for the graph distance of two sites lying in the same cluster.
Kubota \cite{K} showed  quenched large deviations for the simple random walks on the supercritical Bernoulli percolations on $\mathbb{Z}^{d}$.
The strategy of proof in \cite{K} is similar to the one in Zerner \cite{Z}, 
which showed large deviations for random walks in random environment.
However,
the configurations  of percolations fluctuate and the random walk has non-elliptic transition probability.
These obstructions were overcame by using \cite{AP} Theorem 1.1.

In this paper,
we show quenched large deviation principles for the simple random walk on a certain class of percolations 
on $\mathbb{Z}^{d}$ with long-range correlations.
Our result is an extension of Kubota's result for the supercritical Bernoulli percolations.
We can apply this result to the model considered by Drewitz, R\'ath, and Sapozhnikov \cite{DRS}. 
The model contains the supercritical Bernoulli site percolations, random interlacements, the vacant set of random interlacements and the level set of the Gaussian free field.
We can also apply this result to the random cluster model up to the slab critical point. 
See Section 2 for detail.

%strategy of proof
Our strategy of proof follows the one in \cite{Z} and \cite{K}.
In \cite{K}, 
the fact that the Bernoulli measure $P_{p}$ is a product measure on the configuration space is essentially used in order to show that the Lyapunov exponent $\alpha_{\lambda}(\cdot)$ is subadditive.
However, in the case under consideration,
a probability measure $\mathbb{P}$ on the configuration space 
is \textit{not} necessarily a product measure.
In order to get over this obstruction,
we use some ergodic theoretical results for commutative transformations, 
specifically, Furstenberg and Katznelson's theorem \cite{FurKa} and Tao \cite{T} Theorem 1.1.
See Section 4 for our proof.

By using the technique, 
we can also show a shape theorem for the chemical distance,
which is an extension of Garet and Marchand \cite{GM} Corollary 5.4.
We briefly discuss this in Section 5.

Now we describe the setting.
We consider both bond and site percolation on $\mathbb{Z}^{d}$, $d \ge 2$. 

Let $E(\mathbb{Z}^{d})$ be the set of edges of the graph $\mathbb{Z}^{d}$. 
We write $|x|_{\infty} = \max_{1 \le i \le d} |x_{i}|$,
and,
$|x|_{1} = \sum_{1 \le i \le d} |x_{i}|$ for $x = (x_{1}, \dots, x_{d}) \in \mathbb{R}^{d}$.
Let us denote by $\omega$ a configuration on the configuration space.
We write $x \leftrightarrow y$
if $x$ and $y$ are in the same open cluster. 
Let $D(x, y)$ be the graph distance on the vertices of open clusters between $x$ and $y$.
If $x$ and $y$ are in different open clusters,
we let $D(x, y) = +\infty$. 
We often call $D$ the chemical distance. 
Let $\theta_{x}$, $x \in \mathbb{Z}^{d}$, be the shifts on 
the configuration space,
that is,
$\theta_{x}(\omega)(\cdot) = \omega(x + \cdot)$.

\begin{Ass}
Let $\mathbb{P}$ be a probability measure on the configuration space. We assume the following conditions : \\
(i) $\mathbb{P}$ is invariant and ergodic with respect to $\theta_{x}$ for any $x \in \mathbb{Z}^{d} \setminus \{0\}$. \\
(ii) $\mathbb{P}$-a.s. $\omega$, 
there exists a unique infinite open cluster $\mathcal{C}_{\infty} = \mathcal{C}_{\infty}(\omega)$.  \\
(iii) There exist constants $c_{1}, c_{2}, c_{3} > 0$
such that for any $x \in \mathbb{Z}^{d}$, 
\[ \mathbb{P}(0 \leftrightarrow x, D(0, x) \ge c_{1}|x|_{1}) \le c_{1} \exp( - c_{2}(\log |x|_{1})^{1+c_{3}}). \]
\end{Ass}

Let the event $\Omega_{0} := \{0 \in \mathcal{C}_{\infty}\}$.
Thanks to (i) and (ii), $\mathbb{P}(\Omega_{0}) > 0$. 
Let $\overline{\mathbb{P}} := \mathbb{P}(\cdot| \Omega_{0})$. \\ 

Let $((X_{n})_{n \ge 0}, (P^{x}_{\omega})_{x \in \mathcal{C}_{\infty}(\omega)})$
be the Markov chain  on $\mathcal{C}_{\infty}(\omega)$ 
whose transition probabilities are given by $P^{x}_{\omega}(X_{0} = x) = 1$,
\[ P^{z}_{\omega}(X_{n+1} = x+e| X_{n} = x) 
= \frac{1}{2d} \, \text{ if } |e|_{1} = 1 \text{ and } x+e \in \mathcal{C}_{\infty}(\omega), \]
and, 
\[ P^{z}_{\omega}(X_{n+1} = x| X_{n} = x) 
= \frac{1}{2d}|\{e^{\prime} : |e^{\prime}|_{1} = 1,  x + e^{\prime} \notin  \mathcal{C}_{\infty}(\omega)\}|,  \]
for any $x, z \in \mathcal{C}_{\infty}(\omega)$. %corrected

Let $H_{y}$ be the first hitting time to $y \in \mathcal{C}_{\infty}$ for $(X_{n})_{n}$.
For $x, y, z \in \mathcal{C}_{\infty}$, we define the Laplace transform of the hitting time by 
 \[ a_{\lambda}(x, y) = a_{\lambda}^{\omega}(x, y) := -\log E^{x}_{\omega}[\exp(-\lambda H_{y})1_{\{H_{y} < +\infty\}}],  \,  \lambda \ge 0. \]
 
Let $x \in \mathbb{Z}^{d} \setminus \{0\}$.  
Let $T_{x} : \Omega \to \mathbb{N} \cup \{+\infty\}$ be the map defined by
$T_{x}(\omega) = \inf \{n \ge 1 : nx \in  \mathcal{C}_{\infty}(\omega) \}$,
where we let $\inf \emptyset = +\infty$.
We define the maps $\Theta_{x} : \Omega_{0} \to \Omega_{0}$
by $\Theta_{x}\omega = \theta_{x}^{T_{x}(\omega)}\omega$.
Due to the  Poincar\'{e} recurrence theorem, 
$\Theta_{x}$ is well-defined up to sets of measure $0$ under $\overline{\mathbb{P}}$.
By using Lemma 3.3 in Berger and Biskup \cite{BB},
$\Theta_{x}$ is invertible measure-preserving and ergodic with respect to $\overline{\mathbb{P}}$. 
Let $T_{x}^{(n)} := \sum_{k=0}^{n-1} T_{x} \circ \Theta_{x}^{k}$.

\begin{Thm}[Existence of the Lyapunov exponents]
Assume that $\mathbb{P}$ satisfies Assumption 1.1.
Let $\lambda \ge 0$. 
Then, 
there exists a function $\alpha_{\lambda}(\cdot)$ 
on $\mathbb{R}^{d}$ 
such that
$\alpha_{\lambda}(0) = 0$ and 
for any $x \in \mathbb{Z}^{d} \setminus \{0\}$,  
\[ \lim_{n \to \infty} \frac{a_{\lambda}(0, T_{x}^{(n)}x)}{T_{x}^{(n)}} 
= \alpha_{\lambda}(x),  \, \overline{\mathbb{P}}  \text{-a.s}.\]
Moreover, 
$\alpha_{\lambda}(\cdot)$ satisfies the following properties : 
for any $x, y \in \mathbb{R}^{d}$
and for any $q \in  (0, +\infty)$,
$\alpha_{\lambda}(qx) = q \alpha_{\lambda}(x)$,
$\alpha_{\lambda}(x+y) \le  \alpha_{\lambda}(x) + \alpha_{\lambda}(y)$, and, 
$\lambda|x|_{1} \le \alpha_{\lambda}(x) \le (\lambda + \log (2d))C\mathbb{P}(\Omega_{0})|x|_{1}$,
where $C$ is a constant which does not depend on $(\lambda, x)$. 
\end{Thm}

%WE DON'T KNOW THAT WHETHER $\alpha_{\lambda}(-x) = \alpha_{\lambda}(x)$ or not. 

$\alpha_{\lambda}(\cdot)$ is called the \textit{Lyapunov exponent}. 
This is an extension of Theorem 1.1 in \cite{K} and this is the key ingredient of the proof of the following result.

\begin{Thm}[Quenched large deviation principles]
Assume that $\mathbb{P}$ satisfies Assumption 1.1.
Then,
the law of $X_{n}/n$ 
obeys
the following large deviation principles 
with rate function
$I(x) = \sup_{\lambda \ge 0} (\alpha_{\lambda}(x) - \lambda)$, $x \in \mathbb{R}^{d}$,
where $\alpha_{\lambda}(\cdot)$ is the function on $\mathbb{R}^{d}$ in Theorem 1.2.\\
$(1)$ Upper bound :
For any closed set $A$ in $\mathbb{R}^{d}$,
we have $\overline{\mathbb{P}}$-a.s. $\omega$,
\[ \limsup_{n \to \infty} \frac{\log P_{\omega}^{0}(X_{n}/n \in A)}{n}
\le - \inf_{x \in A} I(x).\] 
$(2)$ Lower bound :
For any open set $B$ in $\mathbb{R}^{d}$,
we have $\overline{\mathbb{P}}$-a.s. $\omega$,
\[ \liminf_{n \to \infty} \frac{\log P_{\omega}^{0}(X_{n}/n \in B)}{n}
\ge - \inf_{x \in B} I(x).\] 
\end{Thm}

By using Theorem 1.2,
we can show this theorem in the same way as in the proof of Theorem 1.3 in \cite{K},
so we omit the proof.
See \cite{K} and the references therein.

%(6) Structure of paper :  
This paper is organized as follows. 
In Section 2, we give examples of models satisfying Assumption 1.1. 
In Section 3, we give some preliminaries. 
In Section 4, we show Theorem 1.2. 
In Section 5, we discuss a shape theorem for the chemical distance.

%%%%%%%%%%%%%%
%%%%%%%%%%%%%%%
%%%%%%%%%%%%%%%

\section{Examples of models}

In this section, we state two examples of models satisfying Assumption 1.1. 

\subsection{The model considered by Drewitz, R\'ath, and \\ Sapozhnikov}

Drewitz, R\'ath, and Sapozhnikov \cite{DRS}
considered a certain class of percolation models on $\mathbb{Z}^{d}$ with long range correlations. 
They obtained  large deviation estimates for the chemical distance,
which is similar to \cite{AP} Theorem 1.1.
By using the result, 
they also obtained a shape theorem for the chemical distance.
See \cite{DRS} for detail. 

\begin{Prop}
If a family of probability measures $\{P_{u}\}_{a < u < b}$ 
on $\{0,1\}^{\mathbb{Z}^{d}}$
satisfies the conditions (P1)-(P3) and (S1)-(S2) in Drewitz, R\'ath, and Sapozhnikov's paper, 
then, for each $u \in (a, b)$, 
$P_{u}$ satisfies Assumption 1.1. 
\end{Prop}

\begin{proof}
Assumption 1.1(i) follows from (P1),
Assumption 1.1(ii) follows from (S1) and (S2),
and,
Assumption 1.1(iii) follows from Theorem 1.3 in \cite{DRS}. 
\end{proof}

\subsection{The random-cluster model}

Now we state our setting.  
See Grimmett's book \cite{G2} 
for basic definitions and properties of the random-cluster model.  
Let $d \ge 2$, $p \in [0,1]$ and $q \ge 1$. 
Let $\mathbb{P}_{\Lambda, p, q}^{\xi}$ be the random-cluster measure
on a box $\Lambda$ in $\mathbb{Z}^{d}$ 
with boundary condition $\xi \in \{0,1\}^{E(\mathbb{Z}^{d})}$. 
Let $\mathbb{P}_{p,q}^{b}$, $b = 0,1$, be the extremal infinite-volume limit random-cluster measures. %with free or wired conditions. 

Let $p_{c}^{b}(q) = \inf \{p \in [0,1] : \mathbb{P}_{p,q}^{b}(0 \leftrightarrow \infty) > 0\}$, $b  = 0, 1$.
Then, $p_{c}^{0}(q) = p_{c}^{1}(q)$ and we write this as $p_{c}(q)$. 
We have $p_{c}(q) \in (0,1)$. 
For any $p > p_{c}(q)$,
there exists a unique infinite cluster $\mathcal{C}_{\infty}$, $\mathbb{P}_{p,q}^{b}$-a.s.

We define the \textit{slab critical point} $\hat{p}_{c}(q)$ as follows : 
If $d \ge 3$, 
we let 
\[ S(L, n) := [0, L-1] \times [-n, n]^{d-1}, \] 
\[ \hat{p}_{c}(q, L) := \inf \left\{ p : \liminf_{n \to \infty} \inf_{x \in S(L,n)} \mathbb{P}_{S(L,n), p,q}^{0}(0 \leftrightarrow x) > 0 \right\}, \]
and, 
\[ \hat{p}_{c}(q) := \lim_{L \to \infty} \hat{p}_{c}(q, L). \]
If $d = 2$,  we let 
\[ p_{g}(q) := \sup \left\{p : \lim_{n \to \infty} \frac{- \log \mathbb{P}_{p,q}^{0}(0 \leftrightarrow e_{n})}{n} > 0 \right\}, \, \, \,  e_{n} = (n, 0, \dots, 0) \in \mathbb{R}^{d}, \]
and,   
\[ \hat{p}_{c}(q) := \frac{q(1-p_{g}(q))}{p_{g}(q) + q(1 - p_{g}(q))}.\]

It is known that $\hat{p}_{c}(q) \ge p_{c}(q)$.

\begin{Prop}
If $p > \hat{p}_{c}(q)$,
then,  
$\mathbb{P}_{p, q}^{b}$, $b=0,1$, satisfies Assumption 1.1. 
\end{Prop}

\begin{proof}
It is well-known that 
$\mathbb{P}_{p, q}^{b}$, $b=0,1$, 
satisfies Assumption 1.1 (i) for all $p \in [0,1]$, and, 
Assumption 1.1 (ii) for $p > p_{c}(q)$.

%In fact, (i) follows from \cite{G2} (4.19) and (4.23), (ii) follows from (5.99) and the def of p_{c}(q)
Now we show Assumption 1.1(iii). 
%We follow the strategy taken in the proof of \cite{AP} Theorem 1.1. 
%The key point is to show the random-cluster version of (2.14) in \cite{AP}. 
%The differences are the following inequality (2.1).   
%Assume $d \ge 2$, $q \ge 1$, and $p > \hat{p}_{c}(q)$. 
Let $\Omega := \{0,1\}^{E(\mathbb{Z}^{d})}$ 
and $\Omega^{\prime} := \{0,1\}^{\mathbb{Z}^{d}}$.
Let $B_{0}(r) := [-r, r]^{d}$, $r \ge 0$. 
Let $B_{i}(N) := \tau_{(2N+1)i} B_{0}(N)$ and $B_{i}^{\prime}(N) := \tau_{(2N+1)i} B_{0}(5N/4)$,
$i \in \mathbb{Z}^{d}$,  
where $\tau_{i}$ is the transformation on $\mathbb{Z}^{d}$ defined by $\tau_{i}(x) := i + x$.
Let $Y_{z} : \Omega^{\prime} \to \{0,1\}$ be the projection mapping to the coordinate $z \in \mathbb{Z}^{d}$.  
 
Let $R_{i}^{(N)}$ be the event in $\Omega$ satisfying the following conditions (i) - (iii) : \\
(i) There exists a unique crossing open cluster for $B_{i}^{\prime}(N)$.\\ 
(ii)  The cluster in (i) intersects all boxes with diameter larger than $N/10$.\\ 
(iii) All open clusters with diameter larger than $N/10$ are connected in $B_{i}^{\prime}(N)$. 

Let 
$\phi_{N} : \Omega \to \Omega^{\prime}$ be the map defined by $(\phi_{N}\omega)_{i} := 1_{R_{i}^{(N)}}(\omega)$, $i \in \mathbb{Z}^{d}$. 
Let 
$\mathbb{P}_{p,q,N}^{b}$ be the image measure of $\mathbb{P}_{p,q}^{b}$ by $\phi_{N}$. 
Let $\mathbb{P}^{*}_{p}$ be the Bernoulli measure on $\Omega^{\prime}$ with parameter $p$. 
By using Pisztora \cite{P} Theorem 3.1 for $d \ge 3$ 
and 
Couronn\'e and Messikh \cite{CM} Theorem 9 for $d = 2$,
we see that  
there exist constants $c_{1}^{\prime}, c_{2}^{\prime} > 0$ depending only on $(d,p,q)$ 
such that for any $N \ge 1$ and $i \in \mathbb{Z}^{d}$,   
\[ \sup_{\xi \in \Omega} \mathbb{P}_{B_{i}^{\prime}(N), p, q}^{\xi}\left((R_{i}^{(N)})^{c}\right) 
\le c_{1}^{\prime} \exp(-c_{2}^{\prime} N). \]
This inequality corresponds to (2.24) in \cite{AP}. 
By using the DLR property,  
\[ \lim_{N \to \infty}  \sup_{z \in \mathbb{Z}^{d}} 
\mathrm{ esssup } \, 
\mathbb{P}_{p,q,N}^{b}\left(Y_{z} = 0| \sigma(Y_{x} :  |x-z|_{\infty} \ge 2)\right) 
= 0. \]
By using Liggett, Schonmann and Stacey \cite{LSS} Theorem 1.3,
we see that  
there exists a function $\overline{p}(\cdot)$ 
such that 
$\overline{p}(N) \to 1$ as $N \to \infty$
and 
$\mathbb{P}^{*}_{\overline{p}(N)}$ is dominated by $\mathbb{P}_{p, q, N}$ for each $N$. 
This claim corresponds to (2.14) in \cite{AP}.
The rest of the proof goes in the same way as in the proof of \cite{AP} Theorem 1.1. 
\end{proof}

\begin{Rem}
If $q = 1$, then, $\hat{p}_{c}(1) = p_{c}(1)$
by Grimmett and Marstrand \cite{GrMa}.
%If $q = 2$, then, 
%by Bodineau \cite{B2}, $\mathbb{P}_{p,2} = \mathbb{P}_{p,2}^{0} = \mathbb{P}_{p,2}^{1}$ for $d \ge 2$ and $p \ne p_{c}(2)$. 
%Moreover, 
If $d \ge 3$, then, by Bodineau \cite{B1}, $\hat{p}_{c}(2) = p_{c}(2)$. 
If $d = 2$, then, by Beffara and Duminil-Copin \cite{BDC}, $\hat{p}_{c}(q) = p_{c}(q)$ for any $q \ge 1$. 
Therefore, Theorem 1.2 and Theorem 1.3 
hold on the whole supercritical regime 
if $q=1$ (the Bernoulli percolation case), 
$q = 2$ (the FK-Ising case), 
or, $d = 2$.  
\end{Rem}

%%%%%%%%%%%%%%
%%%%%%%%%%%%%%%
%%%%%%%%%%%%%%%
%%%%%%%%%%%%%%
%%%%%%%%%%%%%%%
%%%%%%%%%%%%%%%

\section{Preliminaries}
By noting the strong Markov property of $(X_{n})_{n}$,
\[a_{\lambda}(x, z) \le a_{\lambda}(x, y) + a_{\lambda}(y, z), \, x, y, z \in \mathcal{C}_{\infty}. \tag{3.1}\]

By considering a path from $x$ to $y$ of length $D(x, y)$ in $\mathcal{C}_{\infty}$,  
\[ a_{\lambda}(x, y) \le (\lambda + \log (2d))D(x, y), \, x, y \in \mathcal{C}_{\infty}. \tag{3.2} \]

By using Birkhoff's ergodic theorem and Kac's theorem, 
we see that for any $x \in \mathbb{Z}^{d} \setminus \{0\}$, 
\[ \lim_{n \to \infty} \frac{T_{x}^{(n)}}{n} 
= E_{\overline{\mathbb{P}}}[T_{x}] 
= \mathbb{P}(\Omega_{0})^{-1}, \, \, \overline{\mathbb{P}} \text{ -a.s. and  in } L^{1}(\overline{\mathbb{P}}). \tag{3.3} \]
Here we denote the expectation with respect to $\overline{\mathbb{P}}$ by $ E_{\overline{\mathbb{P}}}$.

We now describe some assertions derived from Assumption 1.1(iii).
The following assertions correspond to Garet and Marchand \cite{GM} Lemma 2.2 and  Lemma 2.4  respectively.
By using Assumption 1.1(iii), 
we can show them in the same manner as in the proof of \cite{GM} Lemma 2.2 and  Lemma 2.4. 
See \cite{GM} for detail.  

\begin{Lem}
Let $\mathbb{P}$ satisfy Assumption 1.1.
Then, 
there exist
$C_{1}, C_{2} > 0$
such that
for any $r \ge 1$ and
for any $y$ with $|y|_{1} \le r$,
\[ \mathbb{P}\left(D(0, y) \le (3r)^{d}, 0 \leftrightarrow y\right)
\le C_{1} \exp(-C_{2} (\log r)^{1+c_{3}}).\]
\end{Lem}

\begin{Lem}
Let $\mathbb{P}$ satisfy Assumption 1.1.
Then, 
there exists $C_{3} > 0$
such that $E_{\overline{\mathbb{P}}}[D(0, T_{x}x)] \le C_{3} |x|_{1}$ for any $x \in \mathbb{Z}^{d}$.
\end{Lem}

%%%%%%%%%%%%%%
%%%%%%%%%%%%%%%
%%%%%%%%%%%%%%%

\section{Proof of Theorem 1.2}

Let
$\alpha_{\lambda}(x) 
:= \mathbb{P}(\Omega_{0})\inf_{n \ge 1}E_{\overline{\mathbb{P}}}[a_{\lambda}(0, T_{x}^{(n)}x)]/n$ 
for $\lambda \ge 0$ and $x \in \mathbb{Z}^{d}$.
They are also obtained by Kingman's subadditive ergodic theorem as the following.

\begin{Prop}
Let $\lambda \ge 0$ and $x \in \mathbb{Z}^{d} \setminus \{0\}$. Then, 
\[ \lim_{n \to \infty} \frac{a_{\lambda}(0, T_{x}^{(n)}x)}{T_{x}^{(n)}} 
= \alpha_{\lambda}(x),  \, \overline{\mathbb{P}}  \text{-a.s}.\]
\end{Prop}

\begin{proof}
Fix $\lambda \ge 0$ and $x \in \mathbb{Z}^{d} \setminus \{0\}$.
Let $W_{m, n} = a_{\lambda}(T_{x}^{(m)}x, T_{x}^{(n)}x)$, $0 \le m < n$.
Then, by using (3.1), (3.2) and Lemma 3.2,
we see that 
$W_{m+1,n+1} = W_{m,n} \circ \Theta_{x}$, 
$W_{0,n} \le W_{0,m} + W_{m,n}$, and, 
$W_{m, n}  \in L^{1}(\overline{\mathbb{P}})$, $0 \le m < n$.
Therefore we can apply Kingman's subadditive ergodic theorem to $\{W_{m,n}\}_{0 \le m < n}$ and obtain  
\[ \lim_{n \to \infty} \frac{a_{\lambda}(0, T_{x}^{(n)}x)}{n} 
= \inf_{n \ge 1} \frac{E_{\overline{\mathbb{P}}}[a_{\lambda}(0, T_{x}^{(n)}x)]}{n}, \,  \, \overline{\mathbb{P}}\text{-a.s}.\]

By using (3.3),
we have that 
\[ \lim_{n \to \infty} \frac{a_{\lambda}(0, T_{x}^{(n)}x)}{T_{x}^{(n)}} 
= \alpha_{\lambda}(x),  \, \overline{\mathbb{P}} \text{-a.s}.\]
\end{proof}

We need the following lemma 
in order to show the subadditivity of the Lyapunov exponents.

\begin{Lem}
Let $z_{1}, z_{2} \in \mathbb{Z}^{d}$.
Then,
\[ \lim_{n \to \infty} \frac{1}{n} \sum_{i=1}^{n} 
\mathbb{P}(\Omega_{0} \cap \theta_{z_{1}}^{-i}\Omega_{0} \cap \theta_{z_{2}}^{-i}\Omega_{0}) 
\text{ exists and is positive.}\]
We denote this limit by $b_{z_{1}, z_{2}}$.
\end{Lem}

\begin{proof}
By using Tao \cite{T} Theorem 1.1, 
there exists a function $g \in L^{2}(\mathbb{P})$ such that 
\[  \frac{1}{n} \sum_{i=1}^{n} (1_{\Omega_{0}} \circ \theta_{0}^{i}) \cdot 
(1_{\Omega_{0}} \circ \theta_{z_{1}}^{i}) \cdot (1_{\Omega_{0}} \circ \theta_{z_{2}}^{i}) \to g, \, n \to \infty, \,  \text{in } L^{2}(\mathbb{P}).\]

Hence, 
\[ \lim_{n \to \infty} \frac{1}{n} \sum_{i=1}^{n} 
\mathbb{P}(\Omega_{0} \cap \theta_{z_{1}}^{-i}\Omega_{0} \cap \theta_{z_{2}}^{-i}\Omega_{0}) 
= \int g d\mathbb{P}. \]

Since $\mathbb{P}(\Omega_{0}) > 0$,
it  follows from Furstenberg and Katznelson's theorem \cite{FurKa}
that
\[ \liminf_{n \to \infty} \frac{1}{n} \sum_{i=1}^{n} 
\mathbb{P}(\Omega_{0} \cap \theta_{z_{1}}^{-i}\Omega_{0} \cap \theta_{z_{2}}^{-i}\Omega_{0}) > 0.\]

These complete the proof.
\end{proof}

\begin{Prop}
Let $x, y \in \mathbb{Z}^{d}$
and $q \in \mathbb{N}$.
Then, we have that \\
$(i)$ $\alpha_{\lambda}(x+y) \le  \alpha_{\lambda}(x) + \alpha_{\lambda}(y)$.\\
$(ii)$ $\alpha_{\lambda}(qx) = q \alpha_{\lambda}(x)$.\\
$(iii)$ $\lambda|x|_{1} \le \alpha_{\lambda}(x) \le (\lambda + \log (2d))C_{3}\mathbb{P}(\Omega_{0})|x|_{1}$,
where $C_{3}$ is the constant in Lemma 3.2.
\end{Prop}

\begin{proof}
We can see the assertion (ii) 
by using the methods taken in the  proof of \cite{K}, Corollary 2.4.

By using (3.2) and Lemma 3.2,
we have that $E_{\overline{\mathbb{P}}}[a(0, T_{x}x)] \le (\lambda + \log (2d))C_{3}|x|_{1}$
and hence
$\alpha_{\lambda}(x) \le (\lambda + \log (2d))C_{3}\mathbb{P}(\Omega_{0})|x|_{1}$.
We see that $\lambda|x|_{1} \le \alpha_{\lambda}(x)$
by using the methods taken in the proof of \cite{K}, Lemma 2.2.
Thus we have the assertion (iii).

Now we show the assertion (i).
We can assume without loss of generality that $x, y, x+y \in \mathbb{Z}^{d} \setminus \{0\}$.
For $z_{1}, z_{2} \in \mathbb{Z}^{d}$, let 
\[ A_{z_{1}, z_{2}} := \{z_{1}, z_{2} \in \mathcal{C}_{\infty}, \, a_{\lambda}(z_{1}, z_{2}) \le c_{1}(\lambda + \log(2d))|z_{1}-z_{2}|_{1} \},\]
where $c_{1}$ is the constant in Assumption 1.1(iii).

Let 
\[ A_{i} := A_{0, ix} \cap A_{0, i(x+y)} \cap A_{ix, i(x+y)}. \]

By using (3.1),
\[ \frac{1}{n} \sum_{i=1}^{n} E_{\overline{\mathbb{P}}}\left[\frac{a_{\lambda}(0, i(x+y))}{i}, A_{i}\right]\]
\[ \le \frac{1}{n} \sum_{i=1}^{n}  E_{\overline{\mathbb{P}}} \left[\frac{a_{\lambda}(0, ix)}{i}, A_{i}\right] + \frac{1}{n} \sum_{i=1}^{n}  E_{\overline{\mathbb{P}}}\left[\frac{a_{\lambda}(ix, i(x+y))}{i}, A_{i}\right].\]

Now it is sufficient to show the following convergences.

\begin{equation}  \lim_{n \to \infty}\frac{1}{n} \sum_{i=1}^{n} E_{\overline{\mathbb{P}}} \left[
\frac{a_{\lambda}(0, i(x+y))}{i}, A_{i} \right] = \alpha_{\lambda}(x+y)\frac{b_{x, x+y}}{\mathbb{P}(\Omega_{0})}. \tag{4.1}
\end{equation}
\begin{equation}  \lim_{n \to \infty}\frac{1}{n} \sum_{i=1}^{n} E_{\overline{\mathbb{P}}}\left[
\frac{a_{\lambda}(0, ix)}{i}, A_{i} \right] = \alpha_{\lambda}(x)\frac{b_{x, x+y}}{\mathbb{P}(\Omega_{0})}. \tag{4.2}
\end{equation}
\begin{equation}    \lim_{n \to \infty}\frac{1}{n} \sum_{i=1}^{n} E_{\overline{\mathbb{P}}}\left[
\frac{a_{\lambda}(ix, i(x+y))}{i}, A_{i}\right] = \alpha_{\lambda}(y)\frac{b_{x, x+y}}{\mathbb{P}(\Omega_{0})}. \tag{4.3}
\end{equation}
Here $b$ denotes the constant in Lemma 4.2.

Now we prepare the following lemma.
\begin{Lem}
\[ \lim_{n \to \infty}\frac{1}{n} \sum_{i=1}^{n} 
\overline{\mathbb{P}}(A_{i}) = \frac{b_{x, x+y}}{\mathbb{P}(\Omega_{0})}.\]
\end{Lem}

\begin{proof}
By using Lemma 4.2, it is sufficient to show that 
\[ \lim_{i \to \infty}
\mathbb{P}(A_{i}^{c} \cap \Omega_{0} \cap \theta_{x}^{-i}\Omega_{0} \cap \theta_{x+y}^{-i}\Omega_{0}) = 0.\]
By noting (3.2) and Assumption 1.1(iii), 
\begin{align*} 
\mathbb{P}\left(A_{i}^{c} \cap \Omega_{0} \cap \theta_{x}^{-i}\Omega_{0} \cap \theta_{x+y}^{-i}\Omega_{0}\right) 
&\le \mathbb{P}\left(\Omega_{0} \cap \theta_{x}^{-i}\Omega_{0} \cap A_{0,ix}^{c}\right) + \mathbb{P}\left(\Omega_{0} \cap \theta_{x+y}^{-i}\Omega_{0} \cap A_{0,i(x+y)}^{c}\right) \\
&\, \, \, \, + \mathbb{P}\left(\theta_{x}^{-i}\Omega_{0} \cap \theta_{x+y}^{-i}\Omega_{0} 
\cap A_{ix,i(x+y)}^{c}\right) \\ 
&\le \mathbb{P}\left(D(0,ix) > c_{1} i|x|_{1}, 0 \leftrightarrow ix \right)\\
&\, \, \, \, + \mathbb{P}\left(D(0,i(x+y)) > c_{1} i|x+y|_{1}, 0 \leftrightarrow i(x+y) \right)\\
&\, \, \, \, + \mathbb{P}\left(D(ix,i(x+y)) > c_{1} i |y|_{1}, ix \leftrightarrow i(x+y) \right)\\
&\le 3 c_{1} \exp\left(-c_{2}(\log (i \min \{|x|_{1}, |x+y|_{1}, |y|_{1}\}))^{1+c_{3}}\right).
\end{align*}
Since $x, y, x+y \ne 0$,
$\exp\left(-c_{2}(\log (i \min \{|x|_{1}, |x+y|_{1}, |y|_{1}\}))^{1+c_{3}}\right) \to 0$, $i \to \infty$.
This completes the proof of Lemma 4.4.
\end{proof}

We show (4.2). 
First, we have that  
\[  E_{\overline{\mathbb{P}}}\left[\frac{a_{\lambda}(0, ix)}{i}, A_{i} \right]
=  E_{\overline{\mathbb{P}}}\left[\frac{a_{\lambda}(0, ix)}{i} - \alpha_{\lambda}(x), A_{i}\right] + \alpha_{\lambda}(x)\overline{\mathbb{P}}(A_{i}). \]

By noting Lemma 4.4, 
it is sufficient to show that
\begin{equation} 
E_{\overline{\mathbb{P}}}\left[\left|\frac{a_{\lambda}(0, ix)}{i} - \alpha_{\lambda}(x)\right|1_{A_{i}}\right] \to 0,  \, \, i \to \infty. \tag{4.4}
\end{equation}
By using Proposition 4.1, 
we have that 
\[ \left|\frac{a_{\lambda}(0, ix)}{i} - \alpha_{\lambda}(x)\right|1_{A_{i}} 
\le \left|\frac{a_{\lambda}(0, ix)}{i} - \alpha_{\lambda}(x)\right|1_{\{0, ix \in \mathcal{C}_{\infty}\}} \to 0, \, i \to \infty, \, \overline{\mathbb{P}} \text{-a.s}.\]
By recalling the definition of $A_{i}$, 
\[ \left|\frac{a_{\lambda}(0, ix)}{i} - \alpha_{\lambda}(x)\right|1_{A_{i}} 
\le c_{1}(\lambda + \log(2d)) + \alpha_{\lambda}(x), \, i \ge 1.\]
By using the Lebesgue convergence theorem,
we obtain (4.4).
Thus (4.2) is shown.

We can show (4.1) in the same manner.

Finally we show (4.3).
By noting Lemma 4.4, 
it is sufficient to show that
\[ E_{\mathbb{P}}\left[\left|\frac{a_{\lambda}(ix, i(x+y))}{i} - \alpha_{\lambda}(y)\right|1_{A_{i}}\right] \to 0,  \, \, i \to \infty. \tag{4.5} \]
Here we denote the expectation with respect to $\mathbb{P}$ 
by $E_{\mathbb{P}}$.

By using the shift invariance of $\mathbb{P}$,
we have 
\[ E_{\mathbb{P}}\left[\left|\frac{a_{\lambda}(ix, i(x+y))}{i} - \alpha_{\lambda}(y)\right|1_{A_{i}}\right] = E_{\mathbb{P}}\left[\left|\frac{a_{\lambda}(0, iy)}{i} - \alpha_{\lambda}(y)\right|1_{\theta_{x}^{i}A_{i}}\right].\]

Now we have that
$a_{\lambda}(0, iy) \le c_{1}(\lambda + \log (2d)) i |y|_{1}$ on $\theta_{x}^{i} A_{i}$.

Hence,
\[ \left|\frac{a_{\lambda}(0, iy)}{i} - \alpha_{\lambda}(y)\right|1_{\theta_{x}^{i}A_{i}}
\le 
c_{1}(\lambda + \log (2d)) |y|_{1} + \alpha_{\lambda}(y). \]
By noting Proposition 4.1, 
\[ \left|\frac{a_{\lambda}(0, iy)}{i} - \alpha_{\lambda}(y)\right|1_{\theta_{x}^{i}A_{i}}
\le  \left|\frac{a_{\lambda}(0, iy)}{i} - \alpha_{\lambda}(y)\right|1_{\{0, iy \in \mathcal{C}_{\infty}\}} \to 0, \, i \to \infty, \mathbb{P}\text{-a.s.} \]

Thus we obtain (4.5) by using the Lebesgue convergence theorem
and hence (4.3) is shown.
These complete the proof of Proposition 4.3(i).
\end{proof}

Now we can easily extend the Lyapunov exponent $\alpha_{\lambda}(\cdot)$ to the function on $\mathbb{R}^{d}$ and then we have Theorem 1.2.

%%%%%%%%%%%%%%
%%%%%%%%%%%%%%%
%%%%%%%%%%%%%%%
%%%%%%%%%%%%%%
%%%%%%%%%%%%%%%
%%%%%%%%%%%%%%%
%%%%%%%%%%%%%%
%%%%%%%%%%%%%%%
%%%%%%%%%%%%%%%

\section{A shape theorem for the chemical distance}

In this section, we briefly discuss a shape theorem for the chemical distance.

\begin{Thm}[Existence of directional constants]
Assume that $\mathbb{P}$ satisfies Assumption 1.1.
Then, 
there exists a non-negative function $\mu(\cdot)$  
on $\mathbb{R}^{d}$ 
such that
$\mu(0) = 0$, and, 
for any $x \in \mathbb{Z}^{d} \setminus \{0\}$,  
\[ \lim_{n \to \infty} \frac{D(0, T_{x}^{(n)}x)}{T_{x}^{(n)}} 
= \mu(x),  \, \overline{\mathbb{P}}  \text{-a.s}.\]
Moreover, 
$\mu(\cdot)$ satisfies the following properties : 
for any $x, y \in \mathbb{R}^{d}$
and for any $q \in  (0, +\infty)$,
$\mu(qx) = q \mu(x)$,
$\mu(x+y) \le  \mu(x) + \mu(y)$, and, 
$|x|_{1} \le \mu(x) \le C_{3}|x|_{1}$,
where $C_{3}$ is the constant in Lemma 3.2.  
\end{Thm}

This is an extension of \cite{GM}, Corollary 3.3. 
By replacing the Lyapunov exponent $a_{\lambda}(\cdot, \cdot)$ 
with the chemical distance $D(\cdot, \cdot)$, and modifying the definition of $A_{z_{1}, z_{2}}$ slightly,  
the proof goes in the same way as in the proof of Theorem 1.2. 

Let $\mathcal{D}$ be the Hausdorff distance on $\mathbb{R}^{d}$. 
For $t > 0$,
we let a random subset $B_{t} := \{x \in \mathcal{C}_{\infty} : D(0,x) \le t\}$ of $\mathcal{C}_{\infty}$
on $\Omega_{0}$. 

\begin{Thm}[Shape theorem]
Assume that $\mathbb{P}$ satisfies Assumption 1.1. 
Then, 
\[ \lim_{t \to +\infty} \mathcal{D}(B_{t}/t, B_{\mu}) = 0, \, \overline{\mathbb{P}} \text{-a.s.},\] 
where we let $B_{\mu} := \{y \in \mathbb{R}^{d} : \mu(y) \le 1\}$ 
for the function $\mu$ in Theorem 5.1. 
\end{Thm}

This assertion is an extension of Corollary 5.4 in Garet and Marchand \cite{GM}. 
Thanks to Assumption 1.1, Lemma 3.1 and $|x|_{1} \le \mu(x)$,  
we can show this in a manner similar to the proof of Theorem 5.3 in \cite{GM}. 
In our case, 
$\mu(x) \ne \mu(-x)$ may happen, but this is a minor difference and does not affect the argument.

Theorem 5.2 holds for the Drewitz, R\'ath and Sapozhnikov model and the random-cluster model 
up to the slab critical point. 
For the Drewitz, R\'ath and Sapozhnikov model, 
Theorem 1.5 in \cite{DRS} also states a shape theorem.
However, our approach is different from the one in \cite{DRS}. 
\cite{DRS} introduces a pseudo-metric, which is equal to the chemical distance on $\mathcal{C}_{\infty}$. 
On the other hand,
we do not use the notion.


\begin{thebibliography}{99}
\bibitem{AP} P. Antal and A. Pisztora, On the chemical distance for supercritical Bernoulli percolation, Ann. Probab. 24 (1996) 1036-1048.
\bibitem{BDC} V. Beffara and H. Duminil-Copin, The self-dual point of the two-dimensional random-cluster model is critical for $q \ge 1$,  Probab. Theory Related Fields 153 (2012) 511-542. 
\bibitem{BB} N. Berger and M. Biskup, Quenched  invariance principles for simple random walk on percolation clusters, Probab. Theory Related Fields 137 (2007), 83-120.
\bibitem{B1} T. Bodineau, Slab percolation for the Ising model,  Probab. Theory and Relat. Fields 132 (2005) 83-118. 
%\bibitem{B2} T. Bodineau, Translation invariant Gibbs states for the Ising model,  Probab. Theory and Relat. Fields 135 (2006) 153-168.  
\bibitem{CM} O. Couronn\'e and R. J. Messikh, Surface order large deviations for 2D FK-percolation and Potts models, Stoc. Proc. Appl. 113 (2004) 81-99. 
\bibitem{DRS} A. Drewitz, B. R\'ath and A. Sapozhnikov, On chemical distances and shape theorems in percolation models with long-range correlations, available at arXiv : 1212.2885v1.
\bibitem{FurKa} H. Furstenberg and Y. Katznelson, An ergodic Szem\'eredi teorem for commuting transformations, J. Analyse. Math. 34 (1978), 275-291.
\bibitem{GM} O. Garet and R. Marchand, Asymptotic shape for the chemical distance and first-passage percolation on the infinite Bernoulli cluster, ESAIM Probab. Stat. 8 (2004), 169-199. 
\bibitem{G} G. R. Grimmett, \textit{Percolation}, Springer-Verlag, 2nd ed., 1999.
\bibitem{G2} G. R. Grimmett, \textit{The random-cluster model}, Springer-Verlag, 2006.
\bibitem{GrMa} G. Grimmett and J. Marstrand, The supercritical phase of percolation is well behaved, Proc. R. Soc. Lond. Ser. A 430 (1990) 439-457.
\bibitem{K} N. Kubota, Large deviations for simple random walk on supercritical percolation clusters, Kodai Math. J. 35 (2012), 560-575. 
\bibitem{LSS} T. Liggett, R. Schonmann and A. Stacey, Domination by product measures, Ann. Probab. 25 (1997) 71-95.  
\bibitem{P} A. Pisztora, Surface order large deviations for Ising, Potts and percolation models,  Probab. Theory Relat. Fields 104 (1996) 427-466. 
\bibitem{T} T. Tao, Norm convergence of multiple ergodic averages for commuting transformations, Ergodic Theory Dynam. Systems, 28 (2008), 657-688.
\bibitem{Z} M. P. W. Zerner, Lyapunov exponents and quenched large deviations for multidimensional random walk in random environment,
Ann. Probab. 26 (1998), 1446-1476.
\end{thebibliography}
\end{document}